\documentclass[12pt]{article}

\usepackage{amsmath}
\usepackage{amsthm}
\usepackage{amsfonts}
\usepackage{amssymb}
\usepackage{enumerate}
\usepackage{setspace}
\usepackage{comment}
\usepackage{graphicx}
\usepackage{float}
\usepackage{tikz}
\usepackage{calc}
\usepackage{pgfplots}
\usepackage{subfigure}
\tikzstyle{vertex}=[circle, draw, inner sep=0pt, minimum size=6pt]
\newcommand{\vertex}{\node[vertex]}
\tikzstyle{fillvertex}=[circle, draw, inner sep=0pt, minimum size=6pt, fill=black]
\newcommand{\fillvertex}{\node[fillvertex]}
\tikzstyle{bigvertex}=[circle, draw, inner sep=0pt, minimum size=10pt]

\usepackage{array}
\newcolumntype{L}[1]{>{\raggedright\let\newline\\\arraybackslash\hspace{0pt}}m{#1}}
\newcolumntype{C}[1]{>{\centering\let\newline\\\arraybackslash\hspace{0pt}}m{#1}}
\newcolumntype{R}[1]{>{\raggedleft\let\newline\\\arraybackslash\hspace{0pt}}m{#1}}

\newtheorem{theorem}{Theorem}[section]
\newtheorem{proposition}[theorem]{Proposition}
\newtheorem{lemma}[theorem]{Lemma}
\newtheorem{corollary}[theorem]{Corollary}

\theoremstyle{definition}
\newtheorem{definition}{Definition}[section]
\newtheorem{conjecture}[theorem]{Conjecture}

\theoremstyle{remark}

\newcommand{\atrel}{\mathrm{Rel}}
\newcommand{\splitRel}{\mathrm{spRel}}

\renewcommand{\Re}{\mathrm{Re}}
\renewcommand{\Im}{\mathrm{Im}}

\begin{document}


\title{On the roots of all-terminal reliability polynomials}
\author{Jason Brown and Lucas Mol} 
\date{}

\maketitle

\begin{abstract}
Given a graph $G$ in which each edge fails independently with probability $q\in[0,1],$ the \textit{all-terminal reliability} of $G$ is the probability that all vertices of $G$ can communicate with one another, that is, the probability that the operational edges span the graph.  The all-terminal reliability is a polynomial in $q$ whose roots (\textit{all-terminal reliability roots}) were conjectured to have modulus at most $1$ by Brown and Colbourn.  Royle and Sokal proved the conjecture false, finding roots of modulus larger than $1$ by a slim margin.  Here, we present the first nontrivial upper bound on the modulus of any all-terminal reliability root, in terms of the number of vertices of the graph.  We also find all-terminal reliability roots of larger modulus than any previously known.  Finally, we consider the all-terminal reliability roots of simple graphs; we present the smallest known simple graph with all-terminal reliability roots of modulus greater than $1,$ and we find simple graphs with all-terminal reliability roots of modulus greater than $1$ that have higher edge connectivity than any previously known examples.
\end{abstract}

\noindent
\textit{Keywords:} graph polynomials, all-terminal reliability, polynomial roots, Brown-Colbourn Conjecture

\noindent
\textit{MSC 2010:} 05C31

\section{Introduction and Background}

Let $G = (V,E)$ be an undirected, finite, loopless, connected (multi)graph in which each edge fails independently with probability $q \in [0,1]$ and vertices are always reliable.  The \textit{all-terminal reliability} of $G,$ denoted $\atrel(G;q),$ is the probability that all vertices of $G$ can communicate with one another; that is, the probability that the operational edges span the graph.  All-terminal reliability is a well-studied model of network robustness, and much research has been carried out on a variety of algorithmic and theoretical aspects of the model including algorithmic complexity, polynomial time algorithms for restricted families, efficient bounding procedures, the existence of optimal graphs, and analytic properties (when the all-terminal reliability of a graph is viewed as a function of $q$).  See  \cite{ColbournBook}, for example, or \cite{BrownColbournSurvey} for a more recent survey on all-terminal reliability.  Note that all-terminal reliability is often studied in terms of $p=1-q,$ the probability that each edge is operational, but our results on all-terminal reliability are easier to state and prove in terms of $q,$ so we deal exclusively in the variable $q$ in this article.

The all-terminal reliability of a connected graph $G$ with edge set $E,$ denoted $\atrel(G;q),$ is always a polynomial in $q$ of degree (at most) $m = |E|$, as a subgraph with operational edges $E'\subseteq E$ arises with probability 
\[
(1-q)^{|E'|}q^{|E|-|E'|}.
\]
Summing this probability over all sets $E'$ for which all vertices of $G$ can communicate gives the all-terminal reliability of $G.$  The fact that the polynomial has degree exactly $m$ will be seen later from the $H$-form of the polynomial.

It is natural to inquire about the nature and location of the roots of all-terminal reliability polynomials, called \textit{all-terminal reliability roots} or \textit{ATR roots} henceforth.  
ATR roots were noted to have modulus at most $1$ (in $q$) for small graphs, and it was conjectured in \cite{BCConjecture} that this was the case for all graphs.  This contrasts sharply with what is known for other graph polynomials, such as  chromatic polynomials \cite{SokalChromaticPoly}, independence polynomials \cite{BrownHickmanNowakowskiRoots}, and domination polynomials \cite{Domination}, where the roots are dense in the complex plane.  Despite some results and generalizations in the affirmative \cite{SokalHalfPlane,WagnerSeriesParallel}, the conjecture for ATR roots was shown to be false in \cite{BCFalse}.  However:
\begin{itemize}
\item The ATR roots provided were only outside of the unit disk by a slim margin; the largest modulus of an ATR root found was approximately $1.04$.
\item The simple graphs with ATR roots outside of the unit disk were quite large, with the smallest having over $1500$ vertices and over $3000$ edges.
\item All of the simple graphs with ATR roots outside of the unit disk had many vertices of degree $2,$ and it is unclear whether all simple graphs with ATR roots outside of the unit disk have such low edge connectivity.
\end{itemize}
Finally, although ATR roots of modulus greater than $1$ were found, no general upper bound on the modulus of an ATR root was given.

In this article, we continue the exploration of the location of ATR roots. In Section \ref{ATRUpperBound}, we find a nontrivial (though non-constant) bound on the modulus of any ATR root of a graph $G$ in terms of the order of $G$ (this extends a weaker result in \cite{DanielleThesis}).  In Section \ref{LargerModulus}, we study graphs with ATR roots of modulus greater than $1$, finding graphs with ATR roots of greater modulus than any previously known.  Finally, in Section \ref{SimpleGraphSection} we consider simple graphs with ATR roots of modulus greater than $1.$  We find a smaller example of a simple graph with ATR roots outside of the unit disk, and we find simple graphs that have ATR roots outside of the unit disk and have much higher edge connectivity than any previously known examples.

We shall need some background on reliability for the next section. For a graph $G$ of order $n$ and size $m$ (that is, with $n$ vertices and $m$ edges), we can express the all-terminal reliability polynomial of $G$ as 
\begin{align*}
\atrel(G;q)=\sum_{i=0}^{m-n+1}F_iq^i(1-q)^{m-i},
\end{align*}
where $F_i$ denotes the number of subsets of $E$ of cardinality $i$ whose removal leaves the graph connected.  The collection of all such subsets of $G$ is called the \textit{cographic matroid} of $G,$ and the sequence $(F_0,\hdots,F_{m-n+1})$ is called the \textit{$F$-vector} of the cographic matroid of $G$ (see \cite{ColbournBook} for more detail).  The generating polynomial of the $F$-vector of $G$ is called the \textit{$F$-polynomial} of $G,$ denoted $F(G;x).$  It is known that one can rewrite the polynomial in its \textit{$H$-form} as
\[
\atrel(G;q)=(1-q)^{n-1}\sum_{k=0}^{m-n+1}H_k q^k.
\] 
The sequence $(H_0,\hdots, H_{m-n+1})$ is called the \textit{$H$-vector} of the cographic matroid of $G$.  Moreover, the generating polynomial
\[
H(G;x)=\sum_{k=0}^{m-n+1}H_kx^k
\]
of the $H$-vector turns out to be an evaluation of the well-known two-variable \textit{Tutte polynomial} (see \cite{MerinoMatroid}):
\begin{align}
T(G;1,x)&=H(G;x).\label{tutteconnection}
\end{align}

There is a relatively new interpretation to the coordinates of the $H$-vector of a cographic matroid that we shall find particularly useful.  We describe the \textit{chip-firing game} that yields this new interpretation.  Let $G=(V,E)$ be a connected multigraph without loops, and let $w$ denote a special vertex of $G$.  A \textit{configuration} of $G$ is a function $\theta:V\rightarrow \mathbb{Z}$ for which $\theta(v)\geq 0$ for all $v\neq w$ and $\theta(w)=-\displaystyle\sum_{v\neq w}\theta(v).$  For $v\neq w,$ the number $\theta(v)$ represents the number of chips on vertex $v.$  We imagine that the special vertex $w$ has infinitely many chips.  In configuration $\theta,$ a vertex $v\neq w$ is \textit{ready to fire} if $\theta(v)\geq \deg(v)$; vertex $w$ is ready to fire if and only if no other vertex is ready.  \textit{Firing} vertex $u$ changes the configuration from $\theta$ to $\theta',$ where 
\[
\theta'(u)=\theta(u)-\deg(u)
\]
and for $v\neq u$
\[
\theta'(v)=\theta(v)+l(u,v),
\]
where $l(u,v)$ is the number of edges between $u$ and $v$ in $G.$  A configuration is \textit{stable} when $\theta(v) < \deg(v)$ for all $v\neq w$; that is, if and only if $w$ is ready to fire.

A \textit{firing sequence} $\Theta=(\theta_0,\theta_1,\hdots, \theta_k)$ is a sequence of configurations in which $\theta_i$ is obtained from $\theta_{i-1}$ by firing one vertex that is ready to fire for each $i\in\{1,\hdots, k\}.$  It is \textit{nontrivial} when $k>0.$  We write $\theta_0\rightarrow\theta_k$ when some nontrivial firing sequence starting with $\theta_0$ and ending with $\theta_k$ exists.  Configuration $\theta$ is \textit{recurrent} if $\theta\rightarrow \theta.$  Stable, recurrent configurations are called \textit{critical}.  For a critical configuration $\theta,$ a \textit{critical sequence} is a legal firing sequence of minimal length that makes $\theta$ recur. Merino \cite{MerinoTutte,MerinoMatroid} proved the following surprising result which connects the critical configurations to all-terminal reliability:

\begin{quotation}
Let $\mathcal{C}$ be the set of all critical configurations of $G.$  For each $v\in V\backslash \{w\},$ let $x_v$ be a variable, and for each $\theta\in\mathcal{C},$ define a monomial 
\[
m_\theta=\prod_{v\in V\backslash\{w\}}x_v^{\deg(v)-1-\theta(v)}.
\]
Then the set $\mathcal{M} = \mathcal{M}_w(G)$ of all such monomials is an \textit{order ideal of monomials}; that is, the set is closed under division.  Further, $\mathcal{M}$ is \textit{pure}; that is, all of the maximal monomials (under division) are of the same degree, which turns out to be $m-n+1.$ Moreover, the number of monomials of degree $i$ in $\mathcal{M}$ is precisely $H_{i}$, the $i$th coordinate of the $H$-vector of the cographic matroid of $G.$  In particular, $H_i$ is strictly positive for all $i\in\{0,\hdots,m-n+1\}.$
\end{quotation}

Finally, we remark that Huh \cite{Huh} recently settled an outstanding conjecture on $H$-vectors of matroids, which implies that the $H$-vector of the cographic matroid of a graph $G$ is always \textit{log-concave}, that is,
\[
H_i^2\geq H_{i-1}H_{i+1}
\]
for all $i\in\{1,\hdots,m-n\}.$

\section{An upper bound on the modulus of any ATR root}\label{ATRUpperBound}

We begin by proving an upper bound on the modulus of any ATR root of a $2$-connected graph.  A key tool in the proof is the well-known \textit{Enestr\"{o}m-Kakeya Theorem} (see \cite{EnestromKakeya}, for example), which states that if $f(x)=\displaystyle\sum_{i=0}^{d} a_{i}x^{i}$ is a polynomial with positive coefficients, then the (complex) roots of $f$ lie in the annulus
\[ 
\mbox{min} \left( \left\{ \frac{a_{i-1}}{a_{i}}  : i = 1,\ldots,d \right\} \right) \leq |z| \leq \mbox{max} \left( \left\{ \frac{a_{i-1}}{a_{i}}: i = 1,\ldots,d \right\} \right).
\]

\begin{theorem}\label{mainrelbound}
If $G$ is a $2$-connected graph of order $n$ then any root $z$ of $\atrel(G;q)$ satisfies $|z|\leq n-1$.  Moreover, if $n\geq 3$ and $G$ has a vertex $w$ with no incident multiple edges (in particular if $G$ is a simple graph), then $|z|\leq n-2$.
\end{theorem}

\begin{proof}
Let $G$ be a $2$-connected graph of order $n$ and size $m.$  We focus on the generating function
\[ 
H(G;z) = \sum_{i=0}^{m-n+1} H_{i} z^{i}
\]
for the $H$-vector of the cographic matroid of $G$.  For the first statement of the theorem it suffices to show that the roots of $H(G;z)$ lie in the disk $|z| \leq n-1,$ as $\atrel(G;q)=(1-q)^{n-1}H(G;q).$ 

It follows directly from Huh's result on the log-concavity of the $H$-vector that
\[
\frac{H_{i-1}}{H_i}\leq \frac{H_i}{H_{i+1}}
\]
for all $i\in\{1,\hdots,m-n\}.$  Thus we have 
\[
\frac{H_{i-1}}{H_i}\leq \frac{H_{m-n}}{H_{m-n+1}}
\]
for all $i\in\{1,\hdots, m-n\}.$
It follows that  
\[ 
\mbox{max}  \left( \left\{ \frac{H_{i-1}}{H_i} : i = 1,\ldots,m-n+1 \right\} \right) = \frac{H_{m-n}}{H_{m-n+1}},
\]
and thus by the Enestr\"om-Kakeya Theorem, the roots of $H(G;z)$ have modulus bounded above by $H_{m-n}/H_{m-n+1}$.  We now bound $H_{m-n}/H_{m-n+1}$ from above.

By the results of Merino mentioned in the previous section, the pure order ideal of monomials $\mathcal{M} = \mathcal{M}_w(G)$ has $H_{i}$ monomials of degree $i$ for all $i\in\{0,\hdots,m-n+1\}.$  Consider the set 
\[ S = \{(m_\theta,x)\colon\ m_\theta \in \mathcal{M} \mbox{ of degree } m-n+1,~x \mbox{ is a variable that divides } m_\theta \}.\]
For each $(m_\theta,x) \in S$, note that $m_\theta/x$ is a monomial of degree $m-n$ in $\mathcal{M}$, and the purity of $\mathcal{M}$ ensures that each monomial of degree $m-n$ in $\mathcal{M}$ appears at least once under this construction.  It follows that $H_{m-n}\leq |S|$, and since $|S| \leq(n-1)H_{m-n+1}$, we have that 
\[ \frac{H_{m-n}}{H_{m-n+1}} \leq n-1,\]
so that the roots of $H(G;z)$ lie in the disk $|z|\leq n-1,$ and we have proven the first statement.

For the second statement, let $G$ be as above with $n\geq 3$ and let $\mathcal{M}=\mathcal{M}_w(G)$ as above where $w$ is now a vertex of $G$ without any incident multiple edges.  It suffices to show that the roots of $H(G;z)$ lie in the disk $|z|\leq n-2.$  Let $\theta$ be a critical configuration with corresponding monomial 
\[
m_\theta=\displaystyle\prod_{v\in V\backslash\{w\}}x_v^{\deg(v)-1-\theta(v)}.
\] 
There exists a neighbour $v$ of $w$ such that $\theta(v)=\deg(v)-1$, as otherwise no vertex is ready to fire after $w$ (when $w$ fires it sends only one chip to each of its neighbours), and hence $x_v\nmid m_\theta.$  Thus every monomial in $\mathcal{M}$ is divisible by at most $n-2$ variables, and the same pair counting argument as in the previous part of the proof demonstrates that 
\[ H_{m-n}\leq (n-2)H_{m-n+1}.\]
By the Enestr\"om-Kakeya Theorem, we conclude that the roots of $H(G;z)$ lie in the disk $|z|\leq n-2.$
\end{proof}

Since the all-terminal reliability of a graph is obviously multiplicative over its blocks, the following corollary is immediate.

\begin{corollary}
Let $G$ be a connected graph of order $n\geq 2$ in which the blocks of maximum order have order $s.$  Then any root $z$ of $\atrel(G;q)$ satisfies $|z|\leq s-1.$  Further, if $s\geq 3$ and every block $B$ of order $s$ has a vertex with no incident multiple edges in $B$ (in particular if $G$ is a simple graph), then $|z|\leq s-2.$ \hfill \qed
\end{corollary}

A consequence of the proof of Theorem \ref{mainrelbound} is that
\[
\frac{H_{m-n}(G)}{H_{m-n+1}(G)}\leq n-2
\]
for any simple graph $G$ of order $n$ and size $m\geq n.$  While this bound is not best possible, we are off by at most a factor of $2,$ as one can show that
\[
\frac{H_{m-n}(K_n)}{H_{m-n+1}(K_n)}=\frac{n-2}{2}
\] 
for all $n$ by considering the critical configurations of the chip-firing game on $K_n.$

\section{More reliability roots outside of the unit disk}

We now turn to providing more examples of graphs with ATR roots outside of the unit disk centred at the origin of the complex plane (referred to simply as the \textit{unit disk} henceforth). 
 
\subsection{Reliability roots of larger modulus}\label{LargerModulus}

Brown and Colbourn investigated the roots of all-terminal reliability polynomials in \cite{BCConjecture} where they made the following conjecture.

\begin{conjecture}[Brown-Colbourn Conjecture]
Let $G$ be a connected graph.  If $z$ is a root of $\atrel(G;q)$ then $|z|\leq 1.$  In other words, ATR roots all lie inside the unit disk.
\end{conjecture}

While Wagner proved that the Brown-Colbourn conjecture is true for series-parallel graphs \cite{WagnerSeriesParallel}, the conjecture was proven false in general by Sokal and Royle \cite{BCFalse}.  However, the largest known modulus of an ATR root is approximately $1.04,$ as noted in \cite{SCRel}.  We improve on this here, finding ATR roots that are almost three times further outside of the unit disk.

We generalize the graphs that were found to have ATR roots outside of the unit disk in \cite{BCFalse}.  For positive integers $m,$ $n,$ $a,$ and $b,$ let $G_{m,n}^{a,b}$ be the graph on $m+n$ vertices defined as follows.  Take disjoint complete graphs $K_m$ and $K_n$ and replace every edge by $a$ edges in parallel (i.e.\ replace each edge with a \textit{bundle} of $a$ edges), and then connect every nonadjacent pair of vertices with $b$ parallel edges.  The graphs $G_{3,2}^{1,2}$ and $G_{2,2}^{6,1}$ are shown in Figure \ref{GmnabDrawing}.  The graph $G_{2,2}^{6,1}$ is the smallest multigraph (in order and size) known to have ATR roots outside of the unit disk, and $G_{3,3}^{1,6}$ has the ATR root with the largest known modulus of approximately 1.04 (see \cite{BCFalse}).  We will see that $G_{n,n}^{1,6}$ has ATR roots of even larger modulus for $n>3.$ The following result gives us a way to compute $\atrel(G_{m,n}^{a,b};q).$

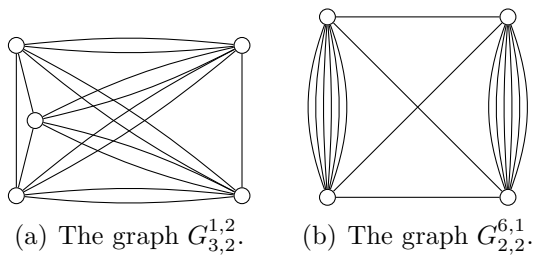
\begin{figure}[h]
\centering
\subfigure[The graph $G_{3,2}^{1,2}.$]{
\centering
\begin{tikzpicture}
\vertex (11) at (0.25,1) {};
\vertex (00) at (0,0) {};
\vertex (02) at (0,2) {};
\vertex (30) at (3,0) {};
\vertex (32) at (3,2) {};
\path
(00) edge (11)
(11) edge (02)
(02) edge (00)
(30) edge (32)
(11) edge[bend left=7] (32)
(11) edge[bend right=7] (32)
(02) edge[bend left=6] (32)
(02) edge[bend right=6] (32)
(00) edge[bend left=6] (32)
(00) edge[bend right=6] (32)
(11) edge[bend left=7] (30)
(11) edge[bend right=7] (30)
(02) edge[bend left=6] (30)
(02) edge[bend right=6] (30)
(00) edge[bend left=6] (30)
(00) edge[bend right=6] (30);
\end{tikzpicture}
}
~
\subfigure[The graph $G_{2,2}^{6,1}.$]{
\begin{tikzpicture}[scale=1.2]
\vertex (11) at (2,2) {};
\vertex (00) at (0,0) {};
\vertex (01) at (0,2) {};
\vertex (10) at (2,0) {};
\path
(00) edge (11)
(00) edge (10)
(01) edge (11)
(01) edge (10)
(10) edge[bend left=20] (11)
(10) edge[bend left=12] (11)
(10) edge[bend left=4] (11)
(10) edge[bend right=20] (11)
(10) edge[bend right=12] (11)
(10) edge[bend right=4] (11)
(00) edge[bend left=20] (01)
(00) edge[bend left=12] (01)
(00) edge[bend left=4] (01)
(00) edge[bend right=20] (01)
(00) edge[bend right=12] (01)
(00) edge[bend right=4] (01);
\end{tikzpicture}
}
\caption{Two examples of the graph $G_{m,n}^{a,b}.$}
\label{GmnabDrawing}
\end{figure}

\begin{proposition}\label{GmnabRecursion}
Let $m,$ $n,$ $a,$ and $b$ be positive integers.  Then
\begin{align}\label{GmnabEquation}
\sum_{i=1}^m\sum_{j=0}^n\tbinom{m-1}{i-1}\tbinom{n}{j}q^{a[i(m-i)+j(n-j)]+b[i(n-j)+j(m-i)]}\atrel\left(G_{i,j}^{a,b};q\right)=1.
\end{align}
\end{proposition}

\begin{proof}
Let $K_m$ and $K_n$ be the complete graphs from which $G_{m,n}^{a,b}$ is formed.  Let $v$ be a vertex of $K_m.$  For a particular subset $C$ of vertices of $G_{m,n}^{a,b}$ which contains $v$ and has $i \geq 1$ vertices from $K_m$ and $j \geq 0$ vertices from $K_n,$ we calculate the probability that $C$ is a connected component in $G_{m,n}^{a,b}$ when edges fail randomly.  In order for $C$ to be a connected component in $G_{m,n}^{a,b}$ when edges fail randomly, all of the vertices of $C$ must be able to communicate with one another and all of the vertices of $C$ must be unable to communicate with any vertex outside of $C.$  The former occurs with probability $\atrel\left(G_{i,j}^{a,b};q\right)$ while the latter occurs with probability
\[
q^{a[i(m-i)+j(n-j)]+b[i(n-j)+j(m-i)]},
\]
as there are $a[i(m-i)+j(n-j)]+b[i(n-j)+j(m-i)]$ edges between $C$ and the remaining vertices of the graph.  A rough sketch of $G_{m,n}^{a,b}$ is provided in Figure \ref{GmnabPicture} to aid in the counting of edges.

\begin{figure}
\centering{
\begin{tikzpicture}
\draw (-1.5,-3) -- (1.5,-3) -- (1.5,3) -- (-1.5,3) -- (-1.5,-3);
\draw (2.5,-3) -- (5.5,-3) -- (5.5,3) -- (2.5,3) -- (2.5,-3);
\draw[dashed] (-1,2.75) -- (5,2.75) -- (5,0.5) -- (-1,0.5) -- (-1,2.75);
\draw (0,-3) node[below] {{$K_m$}};
\draw (4,-3) node[below] {{$K_n$}};
\draw (2,3) node {$C$};
\draw (1,0) node[above]{\scriptsize $a$};
\draw[dotted] (0.9,0) -- (1.1,0);
\draw (3,0) node[above]{\scriptsize $a$};
\draw[dotted] (2.9,0) -- (3.1,0);
\draw (2,-1) node[right]{\scriptsize $b$};
\draw[dotted] (2,-1.1) -- (2,-0.9);
\draw (2,1) node[right]{\scriptsize $b$};
\draw[dotted] (2,1.1) -- (2,0.9);
\draw (2,0) node[right]{\scriptsize $b$};
\draw[dotted] (2,0.1) -- (2,-0.1);
\vertex (v) at (1,1) [label=above:$v$]{};
\vertex (x) at (1,-1) {};
\vertex (y) at (3,1) {};
\vertex (z) at (3,-1) {};
\draw (0.3,2) node {$i$ vertices};
\draw (3.7,2) node {$j$ vertices};
\draw (0,-2) node {$m-i$ vertices};
\draw (4,-2) node {$n-j$ vertices};
\path
(v) edge[bend left=20] (y)
(v) edge[bend right=20] (y)
(v) edge[bend left=20] (z)
(v) edge[bend right=20] (z)
(v) edge[bend left=20] (x)
(v) edge[bend right=20] (x)
(x) edge[bend left=20] (y)
(x) edge[bend right=20] (y)
(x) edge[bend left=20] (z)
(x) edge[bend right=20] (z)
(y) edge[bend left=20] (z)
(y) edge[bend right=20] (z);
\end{tikzpicture}}
\caption{The graph $G_{m,n}^{a,b}$ and a particular subset $C_v$ of vertices.}
\label{GmnabPicture}
\end{figure}
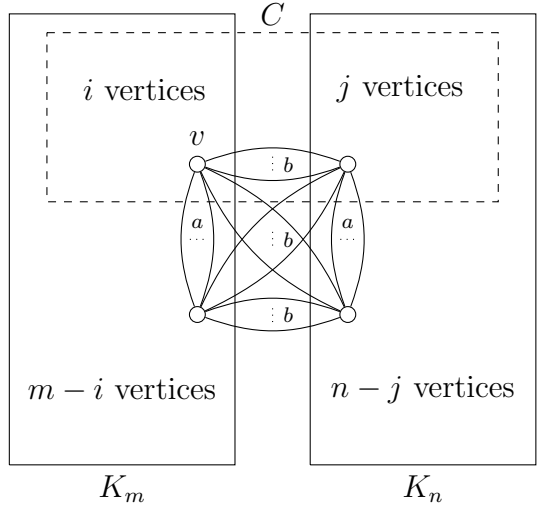

There are $\tbinom{m-1}{i-1}\tbinom{n}{j}$ distinct sets of this form as we may choose any $i-1$ vertices from the remaining $m-1$ vertices of $K_m$ and any $j$ vertices from the $n$ vertices of $K_n.$  Thus the probability that $v$ lies in some connected component containing $i$ vertices from $K_m$ and $j$ vertices from $K_n$ is 
\[
\tbinom{m-1}{i-1}\tbinom{n}{j}q^{a[i(m-i)+j(n-j)]+b[i(n-j)+j(m-i)]}\atrel\left(G_{i,j}^{a,b};p\right)
\]
If we sum over all possibilities for $i$ and $j$ we obtain $1$ as $v$ must be in some component.  This gives (\ref{GmnabEquation}).
\end{proof}

Proposition \ref{GmnabRecursion} gives a recursion for $\atrel\left(G_{m,n}^{a,b};q\right)$ in terms of the smaller polynomials $\atrel\left(G_{i,j}^{a,b};q\right)$ for all $0<i\leq m$ and $0\leq j\leq n$ with $i+j<m+n$, with the base case $\atrel\left(G_{1,0}^{a,b};q\right)=1$ given by the same equation.  This allows us to compute $\atrel\left(G_{m,n}^{a,b};q\right)$ efficiently for small values of $m$ and $n.$  

We numerically computed the ATR roots of the graphs $G_{m,n}^{a,b}$ for all small $m,$ $n,$ $a,$ and $b,$ and graphs of the form $G_{n,n}^{1,6}$ yielded the roots of largest modulus.  While the ATR roots of the graphs $G_{1,1}^{1,6}$ and $G_{2,2}^{1,6}$ all lie inside the unit disk, for each $n\in\{3,4,\hdots,12\}$ the graph $G_{n,n}^{1,6}$ has ATR roots outside of the unit disk.  Table \ref{16Table} shows the ATR roots of $G_{n,n}^{1,6}$ of greatest modulus for $n\in\{3,4,\hdots,12\},$ and the modulus is increasing with $n$ for the values shown.  The ATR roots of largest modulus of $G_{12,12}^{1,6}$ have modulus almost three times as far outside of the unit disk as the best examples from \cite{BCFalse} -- namely, those for the graph $G_{3,3}^{1,6}.$

\begin{table}
\centering
\begin{tabular}{ccc}
$n$ & ATR roots of $G_{n,n}^{1,6}$ of greatest modulus & Modulus\\\hline
3 & $0.6965978094\pm 0.7739344775i$ & $1.0412603341$\\
4 & $0.7225077023\pm 0.7873461471i$ & $1.0686118731$\\
5 & $0.7415248258\pm 0.7932060873i$ & $1.0858337645$\\
6 & $0.7557913447\pm 0.7946437701i$ & $1.0966673507$\\
7 & $0.7665525647\pm 0.7937722633i$ & $1.1034841369$\\
8 & $0.7747703944\pm 0.7917743649i$ & $1.1077796753$\\
9 & $0.7811493576\pm 0.7892664429i$ & $1.1104664951$\\
10 & $0.7861847934\pm 0.7865650322i$ & $1.1121020993$\\
11 & $0.7902223368\pm 0.7838329136i$ & $1.1130343112$\\
12 & $0.7935054014\pm 0.7811532818i$ & $1.1134860896$
\end{tabular}
\caption[ATR roots of $G_{n,n}^{1,6}$ of greatest modulus for small $n.$]{ATR roots of $G_{n,n}^{1,6}$ of greatest modulus for small $n.$  All values rounded to 10 decimal places.}
\label{16Table}
\end{table}

\subsection{Simple graphs with ATR roots outside of the unit disk}\label{SimpleGraphSection}

Of course, for $n\geq 2$ the graphs $G_{n,n}^{1,6}$ discussed in the previous section contain multiple edges.  In \cite{BCFalse}, several \textit{simple} graphs were found that still violated the Brown-Colbourn Conjecture -- the smallest example being the graph on $1512$ vertices and $3016$ edges obtained from $G_{2,2}^{11,1}$ by replacing every edge with $58$ edges in parallel, and then replacing every edge with two edges in series (i.e.\ with a path of length $2$).  The ATR roots of this graph were obtained from the ATR roots of $G_{2,2}^{11,1}$ by transforming to a related generating polynomial and using reduction formulae for the series and parallel edge replacements.  All of the simple graphs in \cite{BCFalse} with ATR roots outside of the unit disk were constructed in a similar manner, and thus they all have edge connectivity $2$ (in fact, they have many vertices of degree $2$).  We improve on these results in two ways: we find a simple graph of smaller order and size that has ATR roots outside of the unit disk, and we find simple graphs with higher edge connectivity that have ATR roots outside of the unit disk.

In order to generate examples of simple graphs with higher edge connectivity that have ATR roots outside of the unit disk, we discuss a more general substitution operation on graphs.  We generalize the idea from \cite{BCFalse} of replacing every edge in a graph with either $k$ edges in parallel or $k$ edges in series.  Essentially, our substitution operation involves replacing every edge in a given graph by \textit{any} fixed graph of our choice.

We define a \textit{gadget} $H(u,v)$ to be a connected graph $H$ of order at least $2$ together with special vertices $u$ and $v$ of $H$ with $u\neq v.$  Let $G$ be a graph and let $H(u,v)$ be a gadget.  An \textit{edge substitution} of the gadget $H(u,v)$ into $G,$ denoted $G[H(u,v)],$ is any graph formed by replacing each edge $\{x,y\}\in E(G)$ by a copy $H_{\{x,y\}}$ of $H,$ identifying $u$ with $x$ and $v$ with $y.$  Note that in order to obtain a specific edge substitution we need to fix an orientation of $G,$ but our results here do not depend on the orientation of $G.$  We let $G[H(u,v)]$ denote \textit{any} edge substitution of the gadget $H(u,v)$ into $G.$

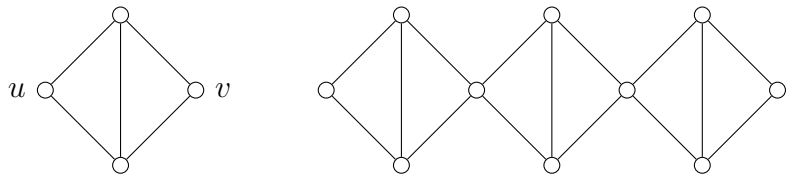
\begin{figure}
\centering{
\begin{minipage}{0.3\textwidth}
\begin{tikzpicture}
\vertex (11) at (1,1) {};
\vertex (00) at (0,0) [label=left:$u$]{};
\vertex (1-1) at (1,-1) {};
\vertex (20) at (2,0) [label=right:$v$]{};
\path
(00) edge (11)
(00) edge (1-1)
(11) edge (1-1)
(11) edge (20)
(1-1) edge (20);
\end{tikzpicture}
\end{minipage}
\begin{minipage}{0.5\textwidth}
\begin{tikzpicture}
\vertex (11) at (1,1) {};
\vertex (00) at (0,0) {};
\vertex (1-1) at (1,-1) {};
\vertex (20) at (2,0) {};
\vertex (40) at (4,0) {};
\vertex (60) at (6,0) {};
\vertex (31) at (3,1) {};
\vertex (3-1) at (3,-1) {};
\vertex (51) at (5,1) {};
\vertex (5-1) at (5,-1) {};
\path
(00) edge (11)
(00) edge (1-1)
(11) edge (1-1)
(11) edge (20)
(1-1) edge (20)
(20) edge (31)
(20) edge (3-1)
(31) edge (3-1)
(31) edge (40)
(3-1) edge (40)
(40) edge (51)
(40) edge (5-1)
(51) edge (5-1)
(51) edge (60)
(5-1) edge (60);
\end{tikzpicture}
\end{minipage}
}
\caption{A gadget $D(u,v)$ and the edge substitution $P_4[D(u,v)]$.}
\label{P4Diamond}
\end{figure}

We will present an expression for the all-terminal reliability of any edge substitution $G[H(u,v)]$ in terms of reliability polynomials of $G$ and $H.$ However, we will need more than just the all-terminal reliability of $H.$  We introduce a new reliability polynomial which we term the \textit{$\{u,v\}$-split reliability}.

\begin{definition}
Let $G$ be a connected graph in which each edge fails independently with probability $q,$ and let $\{u,v\}\subseteq V(G)$ where $u\neq v.$  The \textit{$\{u,v\}$-split reliability} of $G,$ denoted $\splitRel_{\{u,v\}}(G;q)$ is the probability that every vertex $w$ in $G$ can communicate with \textit{exactly} one vertex from $\{u,v\}$ (i.e.\ every vertex in $G$ can communicate with either $u$ or $v$ but not both).
\end{definition}

For example, consider the complete graph $K_4$ on $4$ vertices and let $u$ and $v$ be vertices of $K_4.$  All $10$ operational states for the $\{u,v\}$-split reliability of $K_4$ are pictured in Figure \ref{K4SplitRelStates} -- there are $8$ states with two operational edges and $2$ states with three operational edges.  Hence, the $\{u,v\}$-split reliability of $K_4$ is given by
\[
\splitRel_{\{u,v\}}(K_4;q)=8(1-q)^2q^4+2(1-q)^3q^3.
\]

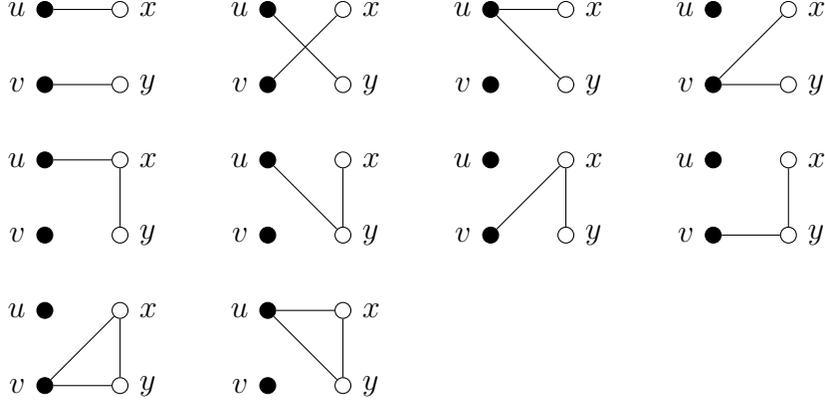
\begin{figure}
\[\arraycolsep=10pt\begin{array}{c c c c}
\begin{tikzpicture}
\fillvertex (u) at (0,1) [label=left:$u$]{};
\fillvertex (v) at (0,0) [label=left:$v$]{};
\vertex (x) at (1,1) [label=right:$x$]{};
\vertex (y) at (1,0) [label=right:$y$]{};
\path
(u) edge (x)
(v) edge (y);
\end{tikzpicture} &
\begin{tikzpicture}
\fillvertex (u) at (0,1) [label=left:$u$]{};
\fillvertex (v) at (0,0) [label=left:$v$]{};
\vertex (x) at (1,1) [label=right:$x$]{};
\vertex (y) at (1,0) [label=right:$y$]{};
\path
(u) edge (y)
(v) edge (x);
\end{tikzpicture} &
\begin{tikzpicture}
\fillvertex (u) at (0,1) [label=left:$u$]{};
\fillvertex (v) at (0,0) [label=left:$v$]{};
\vertex (x) at (1,1) [label=right:$x$]{};
\vertex (y) at (1,0) [label=right:$y$]{};
\path 
(u) edge (x)
(u) edge (y);
\end{tikzpicture} &
\begin{tikzpicture}
\fillvertex (u) at (0,1) [label=left:$u$]{};
\fillvertex (v) at (0,0) [label=left:$v$]{};
\vertex (x) at (1,1) [label=right:$x$]{};
\vertex (y) at (1,0) [label=right:$y$]{};
\path
(v) edge (x)
(v) edge (y);
\end{tikzpicture} \\[10pt]
\begin{tikzpicture}
\fillvertex (u) at (0,1) [label=left:$u$]{};
\fillvertex (v) at (0,0) [label=left:$v$]{};
\vertex (x) at (1,1) [label=right:$x$]{};
\vertex (y) at (1,0) [label=right:$y$]{};
\path
(u) edge (x)
(x) edge (y);
\end{tikzpicture} &
\begin{tikzpicture}
\fillvertex (u) at (0,1) [label=left:$u$]{};
\fillvertex (v) at (0,0) [label=left:$v$]{};
\vertex (x) at (1,1) [label=right:$x$]{};
\vertex (y) at (1,0) [label=right:$y$]{};
\path
(u) edge (y)
(y) edge (x);
\end{tikzpicture} &
\begin{tikzpicture}
\fillvertex (u) at (0,1) [label=left:$u$]{};
\fillvertex (v) at (0,0) [label=left:$v$]{};
\vertex (x) at (1,1) [label=right:$x$]{};
\vertex (y) at (1,0) [label=right:$y$]{};
\path 
(v) edge (x)
(x) edge (y);
\end{tikzpicture} &
\begin{tikzpicture}
\fillvertex (u) at (0,1) [label=left:$u$]{};
\fillvertex (v) at (0,0) [label=left:$v$]{};
\vertex (x) at (1,1) [label=right:$x$]{};
\vertex (y) at (1,0) [label=right:$y$]{};
\path
(v) edge (y)
(x) edge (y);
\end{tikzpicture} \\[10pt]
\begin{tikzpicture}
\fillvertex (u) at (0,1) [label=left:$u$]{};
\fillvertex (v) at (0,0) [label=left:$v$]{};
\vertex (x) at (1,1) [label=right:$x$]{};
\vertex (y) at (1,0) [label=right:$y$]{};
\path
(v) edge (y)
(v) edge (x)
(x) edge (y);
\end{tikzpicture} &
\begin{tikzpicture}
\fillvertex (u) at (0,1) [label=left:$u$]{};
\fillvertex (v) at (0,0) [label=left:$v$]{};
\vertex (x) at (1,1) [label=right:$x$]{};
\vertex (y) at (1,0) [label=right:$y$]{};
\path
(u) edge (y)
(u) edge (x)
(x) edge (y);
\end{tikzpicture}
\end{array}
\]
\caption{The operational states for the $\{u,v\}$-split reliability of $K_4.$}
\label{K4SplitRelStates}
\end{figure}

Now, we can present an expression for the all-terminal reliability of the graph $G[H(u,v)].$  The key is to notice that the \textit{internal} vertices of the gadget (that is, the vertices of $H$ apart from $u$ and $v$) can only communicate with the rest of the graph $G$ through $u$ and $v.$  Thus, in any operational state of $G[H(u,v)]$ each individual copy of the gadget must either be connected, or split between $u$ and $v.$

\begin{proposition}\label{EdgeSubstitution}
Let $G$ be a graph on $n$ vertices and $m$ edges and let $H(u,v)$ be a gadget.  The all-terminal reliability of the graph $G[H(u,v)]$ is given by
\begin{align}
\atrel\left(G[H(u,v)];q\right)=\left[\atrel(H;q)\right]^mF\left(G;\tfrac{\splitRel_{\{u,v\}}(H;q)}{\atrel(H;q)}\right),\label{EdgeSubForm}
\end{align}
where $F(G;z)$ is the $F$-polynomial of $G$.
\end{proposition}

\begin{proof}
Consider any copy of the gadget in any operational state of $G[H(u,v)].$  There are only two possibilities for the gadget if $G[H(u,v)]$ is to be operational: 
\begin{enumerate}[(i)]
\item All of the vertices in the gadget can communicate with one another.  This occurs with probability $\atrel(H;q).$  We say that the gadget is \textit{operational} in this case.
\item Each vertex in the gadget can communicate with exactly one of the vertices $u$ or $v.$  This occurs with probability $\splitRel_{\{u,v\}}(H;q).$  We say that the gadget \textit{splits} in this case.
\end{enumerate}
A gadget splitting in $G[H(u,v)]$ corresponds to an edge failing in $G,$ while an operational gadget in $G[H(u,v)]$ corresponds to an operational edge in $G.$  For any state $\phi$ of $G[H(u,v)],$ let 
\[
E_\phi=\{e\in E(G)\colon\  \mbox{the corresponding gadget $H_e$ is operational}\}.
\]
The state $\phi$ is operational if and only if $E_\phi$ induces a connected subgraph of $G$ and all gadgets corresponding to edges not in $E_\phi$ split.  Thus, the operational states of $G[H(u,v)]$ in which $m-i$ gadgets are operational and $i$ gadgets split correspond exactly to the operational states of $G$ in which $m-i$ edges are operational and $i$ edges fail, and the latter are counted by $F_i$ (from the $F$-vector of the cographic matroid of $G$).  Thus the all-terminal reliability of any edge substitution $G[H(u,v)]$ is given by
\[
\atrel\left(G[H(u,v)];q\right)=\sum_{i=0}^{m-n+1}F_i\left[\atrel(H;q)\right]^{m-i}\left[\splitRel_{\{u,v\}}(H;q)\right]^i,
\]
which can be rewritten as (\ref{EdgeSubForm}) by factoring $[\atrel(H;q)]^m$ out of the sum.
\end{proof}

The expression for $\atrel(G[H(u,v)];q)$ given in Proposition \ref{EdgeSubstitution} allows us to find ATR roots of $G[H(u,v)]$ by a two-step process.  We first find a root $r$ of $\atrel(G;q)$ and then solve a second equation that involves the all-terminal reliability of $H,$ the $\{u,v\}$-split reliability of $H,$ and the root $r.$    

\begin{corollary}\label{EdgeSubstitutionCor}
Let $G$ be a connected graph and let $H(u,v)$ be a gadget.  If $r \neq 1$ is an ATR root of $G,$ then any solution of the equation
\begin{align}\label{SplitSolve}
\splitRel_{\{u,v\}}(H;q)=\tfrac{r}{1-r}\cdot \atrel(H;q)
\end{align}
is either an ATR root of $G[H(u,v)]$ or an ATR root of $H$ itself.
\end{corollary}

\begin{proof}
Let $r\neq 1$ be a root of $\atrel(G;q).$  Since we can write
\[
\atrel(G;q)=(1-q)^mF\left(G;\tfrac{q}{1-q}\right),
\]
$r$ corresponds to the root $\tfrac{r}{1-r}$ of $F(G;x).$  By Proposition \ref{EdgeSubstitution}, 
\begin{align*}
\atrel(G[H(u,v)];q)=\left[\atrel(H;q)\right]^mF\left(G;\tfrac{\splitRel_{\{u,v\}}(H;q)}{\atrel(H;q)}\right)
\end{align*}
for any gadget $H(u,v).$  Therefore, any solution of the equation 
\begin{align}
\frac{\splitRel_{\{u,v\}}(H;q)}{\atrel(H;q)}=\frac{r}{1-r}\label{rational}
\end{align}
is a root of $\atrel(G[H(u,v)];q).$  Suppose now that $z$ is a solution of the equation 
\begin{align}
\splitRel_{\{u,v\}}(H;q)=\frac{r}{1-r}\atrel(H;q).\label{polynomial}
\end{align}
Then either $\atrel(H;z)=0$ or we can divide both sides of (\ref{polynomial}) by $\atrel(H;q)$ to obtain a solution of (\ref{rational}), which implies that $\atrel(G[H(u,v)];z)=0.$
\end{proof}

Using Corollary \ref{EdgeSubstitutionCor} we can find ATR roots of edge substitution graphs by first finding an ATR root $r$ of $G$ and then solving (\ref{SplitSolve}).  Any solutions of (\ref{SplitSolve}) that are not roots of $\atrel(H;q)$ must be ATR roots of $G[H(u,v)].$  An inherent problem with this technique is that we can only solve for the ATR roots of a graph $G$ \textit{exactly} in special cases.  For many graphs we can only \textit{approximate} the ATR roots.  While we can obtain very precise approximations to an ATR root $r$ using numerical methods, we then must solve (\ref{SplitSolve}), and it is well known that the location of the roots of a polynomial can be very sensitive to small changes in the coefficients.  Wilkinson's Polynomial \cite{Wilkinson} is a classic example of this phenomenon.

To get around this problem, instead of solving (\ref{SplitSolve}) numerically, we use a particular stability test due to Schur and Cohn (see \cite{StabilityTest}, Section 11.5) to show that (\ref{SplitSolve}) has solutions outside of the unit disk for all values close to $r,$ so that our numerical approximation to the ATR root $r$ will be sufficient.  The statement of the key result below requires some new notation.  The \textit{complex conjugate} of $a\in\mathbb{C}$ is denoted $\overline{a}$ and the \textit{conjugate transpose} of a complex matrix $A$ is denoted $A^*$ (this is the operation of taking the conjugate of every entry of $A$ and then transposing).  Finally, for a finite sequence $a_1,\hdots, a_n$ of nonzero real numbers, $\mathrm{B}(a_1,\hdots,a_n)$ denotes the number of sign changes in the sequence (i.e.\ the number of indices $k\in\{2,\hdots,n\}$ for which $a_{k-1}a_k<0$).  For example, $\mathrm{B}(-1,1,2,-4,-2)=2.$

\begin{theorem}[Schur-Cohn, \cite{StabilityTest}, Cor.\ 11.5.14] \label{SchurCohnTest}
Let $f(z)=\displaystyle\sum_{k=0}^n a_kz^k$ be a polynomial of degree $n.$  Define the upper triangular matrices
\[
A_k=\begin{bmatrix}
a_0 & a_1 & \hdots & a_{k-1}\\
 & a_0 & \hdots & a_{k-2}\\
 &  & \ddots & \vdots\\
 &  &  & a_0
\end{bmatrix}\ \ \ 
\mbox{ and } \ \ \ 
B_k=\begin{bmatrix}
\overline{a}_n & \overline{a}_{n-1} & \hdots & \overline{a}_{n-k+1}\\
 & \overline{a}_n & \hdots & \overline{a}_{n-k+2}\\
 &  & \ddots & \vdots\\
 &  &  & \overline{a}_n
\end{bmatrix},
\]
where the zero entries have been left blank.  Suppose that for $k\in\{1,\hdots,n\},$ the determinants
\[
M_k=\begin{vmatrix}
B_k^* & A_k\\ A_k^* & B_k
\end{vmatrix}
\]
are all different from zero.  Then $f$ has no root on the unit circle, $\beta=\mathrm{B}(1,M_1,M_2,\hdots,M_n)$ roots outside of the unit circle, and $\alpha=n-\beta$ roots inside it.  \hfill \qed
\end{theorem}

We now build up the particular examples.  One of our goals was to find simple graphs with high edge connectivity that have ATR roots outside of the unit disk, so we will need gadgets with high edge connectivity.  While the complete graph is an obvious candidate, we have found that using the complete graph minus an edge is more effective.  For each $n\geq 3,$ let $K_n^-$ denote the graph obtained from $K_n$ by deleting an edge, and let $K_n^-(u,v)$ denote the gadget in which $u$ and $v$ are nonadjacent in $K_n^-$.  Clearly $K_n^-$ is $(n-2)$-edge-connected.  The following result is straightforward.

\begin{lemma}
Let $G$ be a $2$-edge-connected graph.  For any $n\geq 3,$ the graph $G[K_n^-(u,v)]$ is $(n-1)$-edge-connected. \hfill \qed
\end{lemma}

In order to find ATR roots of an edge substitution $G[K_n^-(u,v)]$ using Corollary \ref{EdgeSubstitutionCor}, we will require formulas for $\atrel(K_n^-;q)$ and $\splitRel_{\{u,v\}}(K_n^-;q).$  We find recursions for $\atrel(K_n^-;q)$ and $\splitRel_{\{u,v\}}(K_n^-;q)$ that are similar to the well-known recursion for $\atrel(K_n;q)$ (see \cite{ColbournBook}, for example):
\[
\atrel(K_n;q)=1-\sum_{i=1}^{n-1}\binom{n-1}{i-1}q^{i(n-i)}\atrel(K_i;q),
\]
with the base case $\atrel(K_1;q)=1.$

\begin{proposition}\label{Recursions}
For any $n\geq 2,$ we have
\begin{align}
\label{Kn-}
\begin{split}
\atrel(K_n^-;q)=&1-\sum_{i=1}^{n-1}\binom{n-2}{i-1}q^{i(n-i)-1}\atrel(K_i;q)\\
&\hspace{0.5cm}-\sum_{i=3}^{n-1}\binom{n-2}{i-2}q^{i(n-i)}\atrel(K_i^-;q), \mbox{ and}
\end{split}\\
\splitRel_{\{u,v\}}(K_n^-;q)=&\sum_{i=1}^{n-1}\binom{n-2}{i-1}q^{i(n-i)-1}\atrel(K_{i};q)\atrel(K_{n-i};q), \label{spKn-}
\end{align}
where $u$ and $v$ are nonadjacent in $K_n^-$.
\end{proposition}

\begin{proof}
The $i$th term in the first sum on the right-hand side of (\ref{Kn-}) gives the probability that the vertex $u$ cannot communicate with $v$ but can communicate with exactly $i$ vertices in $K_n^-$ (including itself).  The $i$th term in the second sum on the right-hand side of (\ref{Kn-}) gives the probability that the vertex $u$ can communicate with $v$ and can communicate with exactly $i$ vertices (including itself and $v$).  Subtracting these probabilities from $1$ for all $i\leq n-1$ leaves the probability that $u$ can communicate with exactly $n$ vertices, i.e.\ the all-terminal reliability of $K_n^-$, and thus (\ref{Kn-}) holds.

For (\ref{spKn-}), in any operational state for the $\{u,v\}$-split reliability of $K_n^-,$ the vertex $u$ must be able to communicate with exactly $i$ vertices (including itself) for some $i\in\{1,\hdots,n-1\},$ while the vertex $v$ must be able to communicate with all of the remaining $n-i$ vertices (including itself).  The probability that $u$ can communicate with exactly $i$ vertices while $v$ can communicate with exactly the remaining $n-i$ vertices is given by
\[
\binom{n-2}{i-1}q^{i(n-i)-1}\atrel(K_{i};q)\atrel(K_{n-i};q).
\]
Summing over $i\in\{1,\hdots,n-1\}$ gives $\splitRel_{\{u,v\}}(K_n^-;q),$ and thus (\ref{spKn-}) holds.
\end{proof}

Our examples of simple graphs with ATR roots outside of the unit disk are all of the form
\[
G^{(k,n)}=G_{3,3}^{k,6k}[K_n^-(u,v)]
\]
for $k\geq 1 $ and $3\leq n\leq 6.$  We start with the base graph $G_{3,3}^{1,6},$ replace every edge with a bundle of $k$ edges, and then substitute the gadget $K_n^-(u,v)$ for every edge.  Before looking at particular examples we outline our general procedure for demonstrating that some $G^{(k,n)}$ has an ATR root outside of the unit disk.

We have found numerically that a particular ATR root $R$ of the graph $G_{3,3}^{1,6}$ satisfies
\begin{align}\label{ReR}
0.69659\leq \Re(R)\leq 0.69660
\end{align}
and
\begin{align}\label{ImR}
0.77393\leq \Im(R)\leq 0.77394.
\end{align}
This is one of the ATR roots of $G_{3,3}^{1,6}$ of greatest modulus (its conjugate $\overline{R}$ would work just as well).  By Proposition \ref{EdgeSubstitution}, the all-terminal reliability of the graph $G_{3,3}^{k,6k}$ obtained from $G_{3,3}^{1,6}$ by replacing each edge with a bundle of $k\geq 1$ edges is given by
\[
\atrel\left(G_{3,3}^{k,6k};q\right)=\atrel\left(G_{3,3}^{1,6};q^k\right),
\]
as the all-terminal reliability of a bundle of $k$ edges is $1-q^k$ and the split reliability is $q^k.$  This means that $\sqrt[k]{R}$ is an ATR root of $G_{3,3}^{k,6k}.$

Now by Corollary \ref{EdgeSubstitutionCor}, any solution of the equation 
\begin{align}\label{sqrtRequation}
\splitRel_{\{u,v\}}(K_n^-;q)=\frac{\sqrt[k]{R}}{1-\sqrt[k]{R}}\cdot\atrel(K_n^-;q)
\end{align}
must be an ATR root of $G^{(k,n)}=G_{3,3}^{k,6k}[K_n^-(u,v)]$ or an ATR root of $K_n^-.$  We have verified that all ATR roots of $K_n^-$ lie inside the unit disk for $3\leq n\leq 6,$ so we may conclude that any solution of (\ref{sqrtRequation}) that lies outside of the unit disk must be an ATR root of $G^{(k,n)}.$  Note that both $\splitRel_{\{u,v\}}(K_n^-;q)$ and $\atrel(K_n^-;q)$ have a factor of $(1-q)^{n-2},$ so we may consider the equation
\begin{align} \label{KnMinusEdge}
\frac{\splitRel_{\{u,v\}}(K_n^-;q)}{(1-q)^{n-2}}=\frac{\sqrt[k]{R}}{1-\sqrt[k]{R}}\cdot\frac{\atrel(K_n^-;q)}{(1-q)^{n-2}}
\end{align}
instead.  The bounds (\ref{ReR}) and (\ref{ImR}) on the real and imaginary parts of the original root $R$ of $G_{3,3}^{1,6}$ translate to bounds on the real and imaginary parts of $\frac{\sqrt[k]{R}}{1-\sqrt[k]{R}}.$  For any real numbers $a$ and $b$ satisfying these bounds (respectively), we apply Theorem \ref{SchurCohnTest} to the polynomial
\[
f_n(q)=\frac{\splitRel_{\{u,v\}}(K_n^-;q)}{(1-q)^{n-2}}-(a+bi)\cdot\frac{\atrel(K_n^-;q)}{(1-q)^{n-2}}.
\]
We are able to determine the sign of all of the required determinants using only the bounds on $a$ and $b.$  By considering the sequence of determinants, we will conclude that there are solutions of (\ref{KnMinusEdge}) that lie outside of the unit disk,  which must be ATR roots of $G[K_n^-(u,v)].$

In the case that $n=3,$ note that $K_3^-\cong P_3$ so that the graph 
\[
G^{(k,3)}=G_{3,3}^{k,6k}[K_3^-(u,v)]
\] 
is obtained from $G_{3,3}^{1,6}$ by only parallel and series substitutions.  Hence this graph is constructed in a similar manner to the smallest simple graph found to have an ATR root outside of the unit disk in \cite{BCFalse}.  The main difference is the choice of the base graph $G_{3,3}^{1,6}$ here instead of $G_{2,2}^{11,1}.$

\begin{proposition}\label{SimpleGraph}
The simple graph $G^{(9,3)}=G_{3,3}^{9,54}[K_3^-(u,v)]$ on $546$ vertices and $1080$ edges has an ATR root outside of the unit disk.
\end{proposition}

\begin{proof}
Recall that a particular ATR root $R$ of the graph $G_{3,3}^{1,6}$ satisfies (\ref{ReR}) and (\ref{ImR}).  From these bounds we are able to obtain
\[
-1.01749\leq \Re\left(\frac{\sqrt[9]{R}}{1-\sqrt[9]{R}}\right)\leq -1.01731
\]
and
\[
10.70762\leq \Im\left(\frac{\sqrt[9]{R}}{1-\sqrt[9]{R}}\right)\leq 10.70814.
\]
Thus, it suffices to show that the polynomial
\begin{align}
f_3(q)=\frac{\splitRel_{\{u,v\}}(K_3^-;q)}{1-q}-(a+bi)\cdot\frac{\atrel(K_3^-;q)}{1-q}.\label{f3q}
\end{align}
has a root outside of the unit disk for \textit{any} real numbers $a$ and $b$ satisfying 
\begin{align}\label{abBounds3}
-1.01749\leq a\leq -1.01731
\ \ \ \mbox{ and } \ \ \ 
10.70762\leq b\leq 10.70814.
\end{align}

Using the recursions of Proposition \ref{Recursions}, we find $\atrel(K_3^-;q)=(1-q)^2$ and $\splitRel_{\{u,v\}}(K_3^-;q)=2q(1-q).$  Substituting these expressions into (\ref{f3q}), we obtain 
\begin{align*}
f_3(q)&=\left(2+a+bi\right)q-(a+bi).
\end{align*}
Applying the test of Theorem \ref{SchurCohnTest} to $f_3(q)$ with $a$ and $b$ as parameters, we get the single determinant
\[
M_1=4a+4.
\]
In particular, from (\ref{abBounds3}) we know that $a<-1,$ and hence $M_1<0.$  Therefore, the number of sign changes $\mathrm{B}(1,M_1)=1,$ and we conclude that $f_3(q)$ has a root outside of the unit disk for any $a$ and $b$ satisfying (\ref{abBounds3}).
\end{proof}

Since the graph $G_{3,3}^{9,54}[K_3^-(u,v)]$ has $546$ vertices and $1080$ edges, it is just over one third of the order and size of the smallest previously known simple graph with ATR roots outside of the unit disk (the smallest such graph found in \cite{BCFalse} has 1512 vertices and 3016 edges).  We stress that the only real difference between our graph and the graph from \cite{BCFalse} in the choice of the graph that we start from before performing the edge substitutions.  We tested many different base graphs of the form $G_{m,n}^{a,b}$ and the base graph $G_{3,3}^{1,6}$ produced the smallest simple graph with ATR roots outside of the unit disk.

It may seem as though our use of Theorem \ref{SchurCohnTest} in Proposition \ref{SimpleGraph} is a little bit heavy-handed, as $f_3(q)$ is a linear function, and we could have verified that its single root is outside of the unit disk directly.  However, Theorem \ref{SchurCohnTest} plays a much more important role in the proof of the following proposition (we omit some details).

\begin{proposition}\label{2EdgeConnected}
The $3$-edge-connected simple graph 
\[
G^{(7,4)}=G_{3,3}^{7,42}[K_4^-(u,v)]
\]
has an ATR root outside of the unit disk.
\end{proposition}

\begin{proof}
Using an argument similar to that of the previous Theorem, it suffices to show that the polynomial
\begin{align*}
f_4(q)&=\frac{\splitRel_{\{u,v\}}(K_4^-;q)}{(1-q)^2}-(a+bi)\cdot\frac{\atrel(K_4^-;q)}{(1-q)^2}\\
&=2(3q+1)q^2-(a+bi)(1-q)(4q^2+3q+1)
\end{align*}
has a root outside of the unit disk for any real numbers $a$ and $b$ satisfying
\begin{align}\label{abBounds4}
-0.90269\leq a\leq -0.90254 \ \ \ \mbox{ and } \ \ \ 8.32420 \leq b\leq 8.32462.
\end{align}

Applying Theorem \ref{SchurCohnTest} to $f_4(q)$ with $a$ and $b$ as parameters, we compute the determinants $M_1,$ $M_2,$ and $M_3$ and find that $M_1,M_2>0$ and $M_3<0,$ so that $\mathrm{B}(1,M_1,M_2,M_3)=1,$ and therefore $f_4(q)$ has exactly one solution outside of the unit disk for any $a$ and $b$ satisfying (\ref{abBounds4}).
\end{proof}

Using the same procedure as in the proof of Proposition \ref{2EdgeConnected}, we can demonstrate that there are $4$-edge-connected and $5$-edge-connected graphs with ATR roots outside of the unit disk.

\begin{proposition}
Both the $4$-edge-connected simple graph $G^{(6,5)}$ and the $5$-edge-connected simple graph $G^{(6,6)}$ have an ATR root outside of the unit disk. \hfill\qed
\end{proposition}

All of the important information about the simple graphs we have found with ATR roots outside of the unit disk is collected in Table \ref{HighConnectivityTable}.  While we suspect that our technique could be used to prove that there are simple graphs with yet higher edge connectivity with ATR roots outside of the unit disk, the time required to apply Theorem \ref{SchurCohnTest} to the polynomial $f_n(q)$ grows large very quickly; after all, the degree of $f_n(q)$ is $\binom{n-1}{2}.$  

\begin{table}
\centering
\begin{tabular}{c C{3cm} C{4cm} C{2cm} C{2cm}}
$n$ & Edge connectivity of $G^{(k,n)}$ & Smallest value $k$ for which $G^{(k,n)}$ has an ATR root outside of the unit disk & Number of vertices of $G^{(k,n)}$ & Number of edges of $G^{(k,n)}$\\\hline
3 & 2 & 9 & 546 & 1,080\\
4 & 3 & 7 & 846 & 2,100\\
5 & 4 & 6 & 1,086 & 3,240\\
6 & 5 & 6 & 1,446 & 5,040
\end{tabular}
\caption{Simple graphs with ATR roots outside of the unit disk.}
\label{HighConnectivityTable}
\end{table}

\vspace{0.2in}

Finally, we mention that our technique for finding ATR roots of simple graphs with high edge connectivity has another important application.  Let $P=\{z\in\mathbb{C}\colon\ \mathrm{Re}(z)\geq -1/2\}.$  It was proven in \cite{BCConjecture} that roots of the $F$-polynomial are dense in $P.$  Hence, for any $r\in P$ and any gadget $H(u,v),$ the solutions of
\begin{align}\label{spRel-Rel}
\splitRel_{\{u,v\}}(H;q)=r\cdot \atrel(H;q)
\end{align}
are limits of ATR roots by Corollary \ref{EdgeSubstitutionCor}, as the roots of a polynomial are a continuous function of its coefficients, and there are graphs whose $F$-polynomials have a root arbitrarily close to $r.$  In particular, when $r=0$ we have
\[
\splitRel_{\{u,v\}}(H;q)=0,
\]
which means that the roots of the $\{u,v\}$-split reliability of any gadget are limits of ATR roots.  For small $n$, the polynomial $\splitRel_{\{u,v\}}(G_{n,n}^{1,6};q)$ has roots outside of the unit disk for any choice of $u$ and $v.$  While none of these roots have modulus larger than the root of largest modulus of the corresponding polynomial $\atrel(G_{n,n}^{1,6};q),$ this demonstrates that roots of the $\{u,v\}$-split reliability polynomial might be useful for finding ATR roots of large modulus.

If we are more creative with our choice of $r,$ we can sometimes find roots of (\ref{spRel-Rel}) that are larger in modulus than any ATR root of $H$ alone.  For example, let $r=-1/2+3i,$ let $H=G_{12,12}^{1,6},$ and let $u$ and $v$ be any two vertices from the same $K_{12}$ in the construction of $H.$  Then (\ref{spRel-Rel}) has roots of modulus slightly larger than any roots of $\atrel(H;q),$ meaning that there are ATR roots of even larger modulus than those we found in Section \ref{LargerModulus}!

Our discussion in the previous two paragraphs shows that there is no inequality relation in general between the ATR roots of largest modulus of $H$ and $G[H(u,v)]$, as there are graphs $G$ for which $G[H(u,v)]$ has an ATR root of modulus larger than any ATR root of $H,$ and other graphs for which the opposite is true.  We can also demonstrate that there is no inequality relation between the roots of largest modulus of $G$ and $G[H(u,v)]$ in general.  Let $B_k(u,v)$ be the gadget on two vertices $u$ and $v$ with $k$ edges in between them.  For any graph $G$ with an ATR root outside of the unit disk and any $k\geq 2$, the ATR root of $G$ of largest modulus has modulus strictly greater than that of any ATR root of $G[B_k(u,v)]$, as $\atrel(B_k(u,v);q)=\atrel(G;q^k),$ and every ATR root of $G$ is pulled in towards the unit circle by this transformation.  On the other hand, we have found numerically that the graph $G_{3,3}^{18,108}[K_3^-(u,v)]$ has an ATR root of larger modulus than any ATR root of $G_{3,3}^{18,108}$ itself.  Other specific examples that fit this situation are $G_{4,4}^{11,66}[K_3^-(u,v)],$ $G_{5,5}^{9,54}[K_3^-(u,v)],$ and $G_{6,6}^{8,48}[K_3^-(u,v)].$

\section{Conclusion}

In this article, we have established several important results on the roots of all-terminal reliability polynomials.  In Section \ref{ATRUpperBound}, we demonstrated the first general upper bound on the modulus of an ATR root of a graph in terms of the order of the graph.  In Section \ref{LargerModulus}, we found ATR roots of larger modulus than any previously known.  Finally, in Section \ref{SimpleGraphSection}, we found the smallest known simple graph with ATR roots outside of the unit disk, and we found simple graphs with ATR roots outside of the unit disk that have higher edge connectivity than any previously known examples.  Our results suggest several different directions for future research.

\subsection*{ATR roots of largest modulus}

While we have proven a non-constant bound on the modulus of any ATR root, the question of whether ATR roots are bounded in modulus by some constant remains a tantalizing open problem.  With regards to finding roots of larger modulus than any currently known, we suspect that for $n\geq 13$ the graphs $G_{n,n}^{1,6}$ will have ATR roots of even larger modulus than $G_{12,12}^{1,6},$ though the computation becomes too lengthy for us to verify this directly.  What is the limiting behaviour of the modulus of the ATR root of largest modulus of the graph $G_{n,n}^{1,6}$?  We are not even certain that the sequence of moduli need be increasing.

\subsection*{Simple graphs with ATR roots outside of the unit disk}

Our study of simple graphs with ATR roots outside of the unit disk produced a smaller example than any previously known, although it is still rather large (it has 546 vertices and 1080 edges).  What is the smallest simple graph with ATR roots outside of the unit disk?

In addition to finding a smaller simple graph with ATR roots outside of the unit disk, we found simple graphs with edge connectivity as high as $5$ that have ATR roots outside of the unit disk.  This is notable as all previously known examples have edge connectivity $2,$ with many vertices of degree $2.$  We also have good candidates for simple graphs of even higher edge connectivity that have ATR roots outside of the unit disk, although the computations required to prove that the roots are outside of the unit disk become increasingly large.  Are there simple graphs with arbitrarily high edge connectivity that have ATR roots outside of the unit disk?  

Finally, while we produced examples of simple graphs that have ATR roots outside of the unit disk and have \textit{edge} connectivity higher than $2,$ the \textit{vertex} connectivity of all of our examples is still only $2.$  Are there simple graphs with vertex connectivity greater than $2$ that have ATR roots outside of the unit disk?  Every graph formed from an edge substitution by a gadget on 3 or more vertices has vertex connectivity at most $2,$ so that the theory we developed in Section \ref{SimpleGraphSection} does not lend itself well to solving this problem.

\subsection*{Split reliability}

All-terminal reliability and $\{u,v\}$-split reliability are simultaneously generalized by the following notion:  Let $G$ be a graph in which each edge fails independently with probability $q$ and let $K\subseteq V(G).$  The \textit{$K$-split reliability} of $G,$ denoted $\splitRel_K(G;q),$ is the probability that every vertex of $G$ can communicate with \textit{exactly} one vertex from $K.$  When $|K|=1$ this is all-terminal reliability and when $|K|=2$ this is $\{u,v\}$-split reliability.  

The $\{u,v\}$-split reliability has proven itself useful in the study of all-terminal reliability, but we believe that $K$-split reliability could have several applications outside of all-terminal reliability as well and is worthy of study in its own right.  For example, consider a network with a fixed set $K$ of \textit{leader nodes} which give orders or instructions, where we would like all of the other vertices to receive orders from exactly one of the leader nodes (to prevent confusion).  The condition for $K$-split reliability ensures that every node receives orders from exactly one leader node, and thus conflicting orders cannot be given.  Chain of command structures seem to be an obvious application of this concept.

We also note that $K$-split reliability gives a measure of the reliability of disconnected graphs.  Let $G$ be a graph with components $G_1,\hdots, G_k$ and let $K=\{v_1,\hdots,v_k\}$ where $v_i\in V(G_i)$ for $i\in\{1,\hdots,k\}.$  Then 
\[
\splitRel_K(G;q)=\prod_{i=1}^k\atrel(G_i;q).
\]
This seems to be a natural measure of the reliability of a disconnected network with edge failures.

\section*{Acknowledgements}
 
\noindent
The authors wish to thank the anonymous referees for their insightful comments and suggestions.  Research of Jason I.\ Brown is partially supported by grant RGPIN 170450-2013 from Natural Sciences and Engineering Research Council of Canada (NSERC).  Research of Lucas Mol is partially supported by an Alexander Graham Bell Canada Graduate Scholarship from NSERC.

\bibliographystyle{amsplain}


\bibliography{OnTheRootsOfAllTerminalReliabilityBibliography}

\begin{thebibliography}{10}
\expandafter\ifx\csname url\endcsname\relax
  \def\url#1{\texttt{#1}}\fi
\expandafter\ifx\csname urlprefix\endcsname\relax\def\urlprefix{URL }\fi

\bibitem{EnestromKakeya}
N.~Anderson, E.~B. Saff, R.~S. Varga, On the {E}nestr\"om-{K}akeya theorem and
  its sharpness, Linear Algebra Appl. 28 (1979) 5--16.

\bibitem{BCConjecture}
J.~I. Brown, C.~J. Colbourn, Roots of the reliability polynomial, SIAM J.
  Discrete Math. 5 (1992) 571--585.

\bibitem{BrownColbournSurvey}
J.~I. Brown, C.~J. Colbourn, Network reliability, in Handbook of the Tutte
  Polynomial (forthcoming).

\bibitem{SCRel}
J.~I. Brown, D.~Cox, The closure of the set of roots of strongly connected
  reliability polynomials is the entire complex plane, Discrete Math. 309~(16)
  (2009) 5043--5047.

\bibitem{BrownHickmanNowakowskiRoots}
J.~I. Brown, C.~A. Hickman, R.~J. Nowakowski, On the location of roots of
  independence polynomials, J. Algebraic Combin. 19~(3) (2004) 273--282.

\bibitem{Domination}
J.~I. Brown, J.~Tufts, On the roots of domination polynomials, Graphs Combin.
  30 (2014) 527--547.

\bibitem{SokalHalfPlane}
Y.-B. Choe, J.~G. Oxley, A.~D. Sokal, D.~G. Wagner, Homogeneous multivariate
  polynomials with the half-plane property, Adv. in Appl. Math. 32~(1--2)
  (2004) 88--187.

\bibitem{ColbournBook}
C.~J. Colbourn, The Combinatorics of Network Reliability, Oxford University
  Press, Inc., New York, NY, USA, 1987.
  
\bibitem{DanielleThesis}
D.~Cox, On network reliability, Ph.D.\ thesis, Dalhousie University, Halifax, Nova Scotia (2013).

\bibitem{Huh}
J.~Huh, {$h$}-vectors of matroids and logarithmic concavity, Adv. Math. 270
  (2015) 49--59.

\bibitem{MerinoTutte}
C.~Merino, Chip-firing and the {T}utte polynomial, Ann. Comb. 1~(3) (1997)
  253--259.

\bibitem{MerinoMatroid}
C.~Merino, The chip firing game and matroid complexes, in: Discrete models:
  combinatorics, computation, and geometry ({P}aris, 2001), Discrete Math.
  Theor. Comput. Sci. Proc., AA, Maison Inform. Math. Discr\`et. (MIMD), Paris,
  2001, pp. 245--255 (electronic).

\bibitem{StabilityTest}
Q.~I. Rahman, G.~Schmeisser, Analytic Theory of Polynomials, Clarendon Press,
  Oxford, 2002.

\bibitem{BCFalse}
G.~Royle, A.~D. Sokal, The {B}rown-{C}olbourn conjecture on zeros of
  reliability polynomials is false, J. Combin. Theory Ser. B 91 (2004)
  345--360.

\bibitem{SokalChromaticPoly}
A.~D. Sokal, Chromatic roots are dense in the whole complex plane, Combin.
  Probab. Comput. 13~(2) (2004) 221--261.

\bibitem{WagnerSeriesParallel}
D.~G. Wagner, Zeros of reliability polynomials and {$f$}-vectors of matroids,
  Combin. Probab. Comput. 9 (2000) 167--190.

\bibitem{Wilkinson}
J.~H. Wilkinson, The evaluation of the zeros of ill-conditioned polynomials.,
  Numer. Math. 1 (1959) 150--180.

\end{thebibliography}

\end{document}